\documentclass{article}
\usepackage[utf8]{inputenc}

\usepackage{amssymb}
\usepackage{array}
\usepackage{authblk}
\usepackage{epsfig}
\usepackage{amsthm}

\usepackage{amsfonts, amssymb, amsmath, amscd, latexsym}

\usepackage{floatflt}
\usepackage{graphicx}
\usepackage{amsmath}
\usepackage{amsfonts}
\usepackage{graphics}
\usepackage{cancel}
\usepackage{multicol}
\usepackage{makeidx}
\usepackage{amssymb}
\usepackage[usenames,dvipsnames]{pstricks}
\usepackage{epsfig}
\usepackage{pst-grad}
\usepackage{pst-plot}
\usepackage{epstopdf,subfigure}
\usepackage[toc,page]{appendix}

\allowdisplaybreaks

\newtheorem{theorem}{Theorem}[section]
\newtheorem{lemma}[theorem]{Lemma}

\newtheorem{definition}[theorem]{Definition}
\newtheorem{exmp}[theorem]{Example}
\newtheorem{rem}[theorem]{Remark}

\def\R{\mathbb R}
\def\C{\mathbb C}
\def\Z{\mathbb Z}
\def\N{\mathbb N}
\def\om{\omega}
\def\ga{\gamma}
\def\dis{\displaystyle}
\DeclareMathOperator{\sech}{sech}

\title{Existence and uniqueness of ($\omega$,c)-periodic solutions of semilinear evolution equations}
\author[1,3]{Makrina Agaoglou\thanks{makrina\_agao@hotmail.com, M.A. is supported by the National Scholarship Programme of the Slovak Republic for the Support of Mobility of Students, PhD Students, University Teachers, Researchers and Artists}}
\author[1,2]{Michal Fe{\v c}kan\thanks{corresponding author, michal.feckan@gmail.com, M.F. is supported by the Slovak
Research and Development Agency (grant number APVV-14-0378) and the Slovak Grant Agency VEGA (grant numbers
2/0153/16 and 1/0078/17)}}
\author[2,4]{Angeliki P. Panagiotidou\thanks{angelipp86@gmail.com}}
\affil[1]{Mathematical Institute of Slovak Academy of Sciences, {\v S}tef{\'a}nikova 49, 814 73 Bratislava, Slovakia}
\affil[2]{Department of Mathematical Analysis and Numerical Mathematics, Comenius University in Bratislava, Mlynsk{\'a} dolina, 842 48 Bratislava, Slovakia}
\affil[3]{Department of Mechanical Engineering, Faculty of Engineering, Aristotle University of Thessaloniki, Thessaloniki 54124, Greece}
\affil[4]{School of Science and Technology, Hellenic Open University, 13-15 Tsamadou str, GR-26222 Patras, Greece}

\begin{document}

\maketitle

\begin{abstract}
In this work we study the existence and uniqueness of $(\omega,c)$-periodic solutions for semilinear evolution equations in complex Banach spaces.
\end{abstract}


{\bf Biographical notes} Makrina Agaoglou is a Postoctoral Researcher at the Aristotle Univerisity of Thessaloniki in Greece and a Researcher in the Slovak Academy of Sciences in Bratislava. She is interested in theoretical and computational dynamical systems and analysis and focuses on applications in mathematical physics and engineering.

Michal Fe\v ckan is a Professor of Mathematics at the Comenius University in Bratislava. He is interested in nonlinear functional analysis, bifurcation theory, dynamical systems and fractional calculus with applications to mechanics, vibrations and economics.

Angeliki Panagiotidou is a Visiting Fellow at the Comenius Univerisity in Bratislava. She is interested in nonlinear partial differential equations and focuses on applications in fluid dynamics, multiphase flows and math biology.

\section{Introduction}

Alvarez et al. \cite{AGP} introduced the concept of $(\omega,c)$-periodic functions by observing that any complex valued solution $x(t)$ of the Mathieu's equation $x''+ax=2q\cos(2t)x$ (see \cite[Chapter 8, Section 4]{CL}) fulfills the equality $x(t+\omega)=cx(t)$ for a complex number $c\in\C$. Note that the Mathieu's equation is the Hill's equation with only one harmonic mode. The Bloch functions, which satisfy the Schr\"odinger equation, have the same property: they are $(\omega,c)$-periodic. Obviously, $(\omega,c)$-periodic functions reduce to the standard $\omega$-periodic functions when $c=1$, and to $\omega$-antiperiodic ones when $c=-1$. These last particular cases are already intensively studied (see \cite{AK,AAD,Far,FNO,H,WXW}).

Motivated by \cite{AGP}, we study the existence and uniqueness of $(\omega,c)$-periodic solutions for semilinear evolution equations in complex Banach spaces. At first, in Section \ref{s1}, we consider that a linear operator of the evolution equation is bounded. Assuming a nonresonance condition, we find a Green function of a nonhomogeneous linear equation with the corresponding boundary value conditions. Then in Section \ref{s2}, we rewrite our problem to a fixed point equation and solving it via the Banach fixed point theorem, we derive an existence and uniqueness result on $(\omega,c)$-periodic solutions. The Schauder fixed point theorem is applied in Section \ref{s3} to prove an existence result for the problem of Section \ref{s2}. We extend in Section \ref{s4} our considerations of Sections \ref{s2} and \ref{s3} to evolution equations with unbounded linear operators. All theoretical results are illustrated by several examples. Related results are studied in \cite{AF,BF,FT}, but this paper deals with more general cases than in the above-mentioned papers.

\section{Preliminary results}\label{s1}
Let $X$ be a complex Banach space with a norm $\|\cdot\|$.
\begin{definition} \cite{AGP}
A function $g : \R\to X$ is called $(\omega,c)$-periodic if there is a pair $(\omega,c)$, where $c\in \C\setminus \{0\}$, $\omega >0$ such that $g(t+\omega)=cg(t)$ for all $t \in \R$.
\end{definition}
We denote by $\Upsilon_{\omega,c}$ the set of all continuous and $(\omega,c)$-periodic functions $g : \R\to X$, and $\Upsilon_{\omega,c}^1=\Upsilon_{\omega,c}\cap C^1(\R,X)$.

Let us consider first the linear equation
\begin{equation}\label{e1}
\dot{y}=Ay
\end{equation}
for a continuous linear mapping $A\in L(X)$. Then its solution is of the form $y(t)=e ^{At}y_{0}$, $t\in \R$, which is $(\omega, c)$-periodic if and only if it holds
$$
\begin{gathered}
y(t+\omega)=cy(t)\Leftrightarrow e ^{At+A\omega}y_{0}=ce ^{At}y_{0}\Leftrightarrow e ^{At}e^{A\omega}y_{0}=ce^{At}y_{0}\\
\Leftrightarrow e ^{A\omega}y_{0}=cy_{0}\Leftrightarrow (cI-e ^{A\omega})y_{0}=0.
\end{gathered}
$$
In this paper we consider the case when $c$ doesn't belong to spectrum $\sigma(e^{A\omega})$ of $e^{A\omega}$, so by the spectral mapping theorem we suppose
\begin{itemize}
\item[(A1)] $c\neq e^{\omega\lambda}$ for all $\lambda\in\sigma(A)$.
\end{itemize}
Now we pass to the nonhomogeneous linear equation
\begin{equation}\label{e2}
\dot{y}=Ay+f(t)
\end{equation}
for $f\in \Upsilon_{\omega,c}$. We present the following simple observation.

\begin{lemma}\label{lem1} $y\in \Upsilon_{\omega,c}$ if and only if it holds
\begin{equation}\label{b1}
y(\omega)=cy(0).
\end{equation}
\end{lemma}
\begin{proof}
Clearly $y\in \Upsilon_{\omega,c}$ implies \eqref{b1}. On the other hand, if \eqref{b1} holds for some $y_0\in C([0,\omega],X)$, then we set
$$
y(t)=c^{k}y_0(t-k\om)
$$
for $t\in[k\om,(k+1)\om]$ and $k\in\Z$. Note
$$
\lim_{t\to k\om_+}y(t)=c^k y_0(0)=c^{k-1}y_0(\om)=\lim_{t\to k\om_-}y(t),
$$
so $y(t)$ is well-defined and continuous. Next, for $t\in[k\om,(k+1)\om]$ and $k\in\Z$, we have $t+\om\in[(k+1)\om,(k+2)\om]$ and
$$
y(t+\om)=c^{k+1}y_0(t+\om-(k+1)\om)=c^{k+1}y_0(t-k\om)=cc^{k}y_0(t-k\om)=cy(t).
$$
So $y\in\Upsilon_{\omega,c}$. The proof is finished.
\end{proof}

Set a Banach space $Z=C([0,\omega],X)$ with the maximum norm $\|y\|_0=\max_{t\in[0,\om]}\|y(t)\|$. Now we are ready to prove the following result.

\begin{lemma}\label{lem2} The solution $y\in Z$ of \eqref{e2} satisfying \eqref{b1} is given by
$$
y(t)=\int_{0}^{\omega}K(t,s)f(s),
$$
where
$$
K(t,s)=\begin{cases}ce^{A(t-s)}(cI-e^{A\omega})^{-1}&\ \text{for}\ s\in[0,t],\\
                    e^{A(\omega+t-s)}(cI-e^{A\omega})^{-1}&\ \text{for}\ s\in(t,\omega],
       \end{cases}
$$
and $I$ is the unit mapping.
\end{lemma}
\begin{proof}
The general solution $y\in Z$ of \eqref{e2} has the form
$$
y(t)=e^{At}y_{0}+\int_{0}^{t}e^{A(t-s)}f(s)ds.
$$
The condition \eqref{b1} gives
$$
e^{A\omega}y_{0}+\int_{0}^{\omega}e^{A(\omega-s)}f(s)ds=cy_{0}\Leftrightarrow  y_{0}=(cI-e^{A\omega})^{-1}\int_{0}^{\omega} e^{A(\omega -s)}f(s)ds,
$$
so the solution of \eqref{e2} satisfying \eqref{b1} is as follows
$$
\begin{gathered}
y(t)=e^{At}y_{0}+\int_{0}^{t} e^{A(t-s)}f(s)ds\\
=e^{At}(cI-e^{A\omega})^{-1}\int_{0}^{\omega} e^{A(\omega -s)}f(s)ds+\int_{0}^{t}e^{A(t-s)}f(s)ds\\
=\int_{0}^{t}e^{A(t-s)}\left((cI-e^{A\omega})^{-1}e^{A\omega}+I\right)f(s)ds\\
+\int_t^\omega e^{A(\omega+t-s)}(cI-e^{A\omega})^{-1}f(s)ds\\
=\int_{0}^{t}ce^{A(t-s)}(cI-e^{A\omega})^{-1}f(s)ds+\int_t^\omega e^{A(\omega+t-s)}(cI-e^{A\omega})^{-1}f(s)ds\\
=\int_{0}^{\omega}K(t,s)f(s)ds,
\end{gathered}
$$
since
$$
(cI-e^{A\om})(cI-e^{A\om})^{-1}=I,
$$
implies
$$
c(cI-e^{A\omega})^{-1}=e^{A\om}(cI-e^{A\omega})^{-1}+I=(cI-e^{A\omega})^{-1}e^{A\om}+I.
$$
The proof is completed.
\end{proof}

\section{A uniqueness result}\label{s2}

In this section we consider the equation
\begin{equation}\label{e4}
\dot{y}=Ay+g(t,y)
\end{equation}
for $g\in C(\R\times X,X)$ satisfying
\begin{itemize}
\item[(C1)] $g(t+\omega,cy)=cg(t,y)$ for all $t\in \R$ and $y\in X$.
\item[(C2)] There is a constant $L>0$ such that $\|g(t,y_{1})-g(t,y_{2})\|\le L\|y_{1}-y_{2}\|$ for all $t\in \R$ and $y_{1},y_{2}\in X$.
\end{itemize}
We are looking for solutions of \eqref{e4} in $\Upsilon_{\om,c}$. First we note (C1) implies if $y\in \Upsilon_{\om,c}$ then $g(t,y(t))\in\Upsilon_{\om,c}$. Then by Lemmas \ref{lem1} and \ref{lem2}, our task is equivalent to the fixed point problem
\begin{equation}\label{fix1}
y(t)=\int_{0}^{\omega}K(t,s)g(s,y(s))ds,\quad y\in Z.
\end{equation}
To solve \eqref{fix1}, we define an operator $S : Z\rightarrow Z$ by
$$
(Sy)(t)=\int_{0}^{\omega}K(t,s)g(s,y(s))ds
$$
for $y\in Z$. Clearly $S$ is well-defined. Next, for $y_{1},y_{2}\in Z$ we derive
$$
\begin{gathered}
\|(Sy_{1})(t)- (Sy_{2})(t)\|\leq \int_{0}^{\omega}\| K(t,s) (g(s,y_{1}(s))-g(s,y_{2}(s)))\| ds\\
\leq \int_{0}^{\omega}\| K(t,s)\|\|g(s,y_{1}(s))-g(s,y_{2}(s))\| ds
\leq L\int_{0}^{\omega}\| K(t,s)\|\|y_{1}(s)-y_{2}(s)\|ds\\
\leq L\| y_{1}-y_{2} \|_{0}\int_{0}^{\omega}\| K(t,s)\|ds\le LM\| y_{1}-y_{2} \|_{0}
\end{gathered}
$$
for
\begin{equation}\label{M}
M=\max_{t\in[0,\om]}\int_{0}^{\omega}\|K(t,s)\|ds.
\end{equation}
Therefore we arrive at the inequality
$$
\| Sy_{1}- Sy_{2}\|_{0}\leq LM\| y_{1}-y_{2}\|_{0}
$$
and from Banach fixed point theorem we get the following result.
\begin{theorem}\label{th1}
Suppose (A1) and consider \eqref{e4} under conditions (C1) and (C2). If
\begin{equation}\label{bf}
LM<1
\end{equation}
for $M$ given by \eqref{M}, then \eqref{e4} has a unique $(\om,c)$-periodic solution $y$ satisfying
\begin{equation}\label{est1}
\|y\|_0\le\frac{M\|g(\cdot,0)\|_0}{1-LM}.
\end{equation}
\end{theorem}
\begin{proof}
The uniqueness and existence result follows from  Banach fixed point theorem along with Lemmas \ref{lem1} and \ref{lem2}. Furthermore, by
$$
\|y\|_0=\|Sy\|_0\le LM\|y\|_0+M\|g(\cdot,0)\|_0,
$$
we have \eqref{est1}, which finishes the proof.
\end{proof}
\begin{rem}
\rm We can estimate $M$ as follows

1.
$$
\begin{gathered}
\int_{0}^{\omega}\|K(t,s)\|ds\\
\le |c|\|(cI-e^{A\omega})^{-1}\| \int_{0}^{t}e^{\| A\| (t-s)}ds+\|e^{A\omega}(cI-e^{A\omega})^{-1}\| \int_{t}^{\omega}e^{\| A\| (t-s)}ds\\
=|c|\|(cI-e^{A\omega})^{-1}\| \frac{e^{\| A \| t}-1}{\| A \|}+\|e^{A\omega}(cI-e^{A\omega})^{-1}\|\frac{1-e^{\| A\| (t-\omega)}}{\| A\|}\\
\leq \frac{e^{\| A \|\omega}-1}{\| A \|}\max\left\{|c|\|(cI-e^{A\omega})^{-1}\|,\frac{\|e^{A\omega}(cI-e^{A\omega})^{-1}\|}{e^{\| A\|\omega}}\right\},
\end{gathered}
$$
so
\begin{equation}\label{Mb}
\begin{gathered}
M\le\frac{e^{\| A \|\omega}-1}{\| A \|}\max\left\{|c|\|(cI-e^{A\omega})^{-1}\|,\frac{\| (cI-e^{A\omega})^{-1}e^{A\omega}\|}{e^{\| A\|\omega}}\right\}
\end{gathered}
\end{equation}
2.
$$
\begin{gathered}
\int_{0}^{\omega}\|K(t,s)\|ds\\
\le |c| \int_{0}^{t}\|e^{A(t-s)}(cI-e^{A\omega})^{-1}\|ds+\int_{t}^{\om}\|e^{A(\om+t-s)}(cI-e^{A\omega})^{-1}\|ds\\
=|c| \int_{0}^{t}\|e^{As}(cI-e^{A\omega})^{-1}\|ds+\int_{t}^{\om}\|e^{As}(cI-e^{A\omega})^{-1}\|ds\\
\leq \max\{|c|,1\}\int_{0}^{\om}\|e^{As}(cI-e^{A\omega})^{-1}\|ds,
\end{gathered}
$$
so
\begin{equation}\label{Mc}
\begin{gathered}
M\le\max\{|c|,1\}\int_{0}^{\om}\|e^{As}(cI-e^{A\omega})^{-1}\|ds.
\end{gathered}
\end{equation}

\end{rem}

Now we can give an example.

\begin{exmp}\label{exp1}
\rm We consider the case $X=\C^2$ for $c=-1$, $\omega=\pi$ and
$$
\begin{gathered}
A=\begin{pmatrix}
2 & -4\\
6 & -8
\end{pmatrix}\\
g(t,y)=(g_{1}(t,y),g_{2}(t,y))=a\left(\sin t\cos (y_{1}+y_{2}),\cos 2t\sin (y_{1}-y_{2})\right)\\
y=(y_1,y_2).
\end{gathered}
$$
Since all parameters are real, we just consider $X=\R^2$. Clearly (C1) holds. Since $\sigma(A)=\{-4,-2\}$, (A1) is satisfied. Next by the Mean value theorem for vector-valued functions
$$
L=\max_{y\in \mathbb{R}^{2}, t\in \mathbb{R}}\| g_{y}(t,y)\|
$$
for
\[g_{y}=\begin{pmatrix}
-a\sin t \sin (y_{1}+y_{2}) & -a\sin t \sin (y_{1}+y_{2})\\
a\cos 2t \cos (y_{1}-y_{2}) & -a\cos 2t \cos (y_{1}-y_{2})
\end{pmatrix}\]
Furthermore, we have
$$
e^{As}(cI-e^{A\omega})^{-1}=\begin{pmatrix}\dis\frac{2e^{4\pi-4s}}{1+e^{4\pi}}-\frac{3}{2}e^{\pi-2s}\sech\pi & e^{\pi-2s}\sech\pi-\frac{2e^{4\pi-4s}}{1+e^{4\pi}}\\
                                           \dis\frac{3e^{4\pi-4s}}{1+e^{4\pi}}-\frac{3}{2}e^{\pi-2s}\sech\pi & e^{\pi-2s}\sech\pi-\frac{3e^{4\pi-4s}}{1+e^{4\pi}}
                                           \end{pmatrix}
$$

Now we consider 3 standard norms on $\R^2$ to estimate $L$ and $M$.

\textbf{Case 1.} $\|y\|_1=|y_1|+|y_2|$. Then we derive
$$
\begin{gathered}
L=\max_{y\in \R^{2}, t\in \R}\| g_{y}(t,y)\|_{1}\\
=|a|\max_{y\in \R^{2}, t\in \R}\left(|\sin t\sin (y_{1}+y_{2})|+|\cos 2t\cos (y_{1}-y_{2})|\right)\\
=|a|\max_{t\in \R}\left(|\sin t|+|\cos 2t|\right)=2|a|
\end{gathered}
$$
and by \eqref{Mc}
$$
M\le1.73883.
$$
So condition \eqref{bf} holds if
\begin{equation}\label{a1}
|a|<0.287549.
\end{equation}

\textbf{Case 2.} $\|y\|_\infty=\max\{|y_1|,|y_2|\}$. Then we derive
$$
\begin{gathered}
L=\max_{y\in \R^{2}, t\in \R}\| g_{y}(t,y)\|_\infty\\
=2|a|\max_{y\in \mathbb{R}^{2},t\in \mathbb{R}}\max\left\{|\sin t\sin (y_{1}+y_{2})|,|\cos 2t\cos (y_{1}-y_{2})|\right\}\\
=2|a|\max_{t\in \R}\max\left\{|\sin t|,|\cos 2t|\right\}=2|a|,
\end{gathered}
$$
and by \eqref{Mc}
$$
M\le1.4907.
$$
So condition \eqref{bf} holds if
\begin{equation}\label{a2}
|a|<0.335414.
\end{equation}

\textbf{Case 3.} $\|y\|_2=\sqrt{y_1^2+y_2^2}$. Then we derive
$$
\begin{gathered}
L=\max_{y\in \R^{2}, t\in \R}\| g_{y}(t,y)\|_2\\
=\sqrt{2}|a|\max_{y\in \mathbb{R}^{2},t\in \mathbb{R}}\max\left\{|\cos 2t\cos(y_1-y_2)|,|\sin t\sin(y_1+y_2)|\right\}\\
=\sqrt{2}|a|\max_{t\in \R}\max\left\{|\sin t|,|\cos 2t|\right\}=\sqrt{2}|a|,
\end{gathered}
$$
and by \eqref{Mc}
$$
M\le1.40635.
$$
So condition \eqref{bf} holds if
\begin{equation}\label{a3}
|a|<0.502795.
\end{equation}
Consequently, the best estimate from \eqref{a1}, \eqref{a2} and \eqref{a3} is \eqref{a3} corresponding the the norm $\|\cdot\|_2$, when there is a unique $\pi$-antiperiodic solution $y(t)$ which is nonconstant.
\end{exmp}

\section{An existence result}\label{s3}

Now we consider instead (C2) the following condition
\begin{itemize}
\item[(C3)] There are constants $g_1\ge0$ and $g_2\ge0$ such that $\|g(t,y)\|\le g_1+g_2\|y\|$ for all $t\in \R$ and $y\in X$.
\end{itemize}
Then like above, we derive
$$
\begin{gathered}
\|(Sy)(t)\|\leq \int_{0}^{\omega}\|K(t,s)g(s,y(s))\| ds\\
\le(g_{1}+g_{2}\|y\|_0)\int_{0}^{\omega}\|K(t,s)\|ds
\leq M(g_{1}+g_{2}\| y\|_0).
\end{gathered}
$$
Therefore from Schauder fixed point theorem we get the following result.
\begin{theorem}\label{th2}
Let $\dim X<\infty$. Suppose (A1) and consider \eqref{e4} under conditions (C1) and (C3). If
\begin{equation}\label{bf2}
g_2M<1
\end{equation}
for $M$ given by \eqref{M}, then \eqref{e4} has a $(\om,c)$-periodic solution $y$ with $\|y\|_0\le \frac{Mg_1}{1-Mg_2}$.
\end{theorem}
\begin{proof}
Set $B(r_0)=\left\{y\in Z\mid \|y\|_0\le r_0\right\}$ for $r_0=\frac{Mg_1}{1-Mg_2}$. Then for any $y\in B(r_0)$ the above computation gives
$$
\|Sy\|_0\leq M(g_{1}+g_{2}\|y\|_0)\le M(g_{1}+g_{2}r_0)=r_0.
$$
Hence $S : B(r_0)\to B(r_0)$. Arzel\`{a}-Ascoli theorem implies the compactness of $S$. Thus Schauder fixed point theorem gives the result. The proof is finished.
\end{proof}

\begin{exmp}\label{exp2}
\rm Consider the problem from Example \ref{exp1}. Then
$$
\|g(t,y)\|_1=|a|\left(|\sin t\cos (y_{1}+y_{2})|+|\cos 2t\sin (y_{1}-y_{2})|\right)\le 2|a|,
$$
so we have $g_1=2|a|$ and $g_2=0$. Consequently, there is a $\pi$-antiperiodic solution $y$ for any $0\ne a\in\R$ with $\|y(t)\|_1\le 3.47767|a|$ for any $t\in\R$. Note when \eqref{a3} holds, we have a unique such a solution.
\end{exmp}

\begin{exmp}\label{exp3}
\rm In this example we consider the case for $c=-1$, $\omega=\pi$ and $g(t,y)=(g_{1}(t,y),g_{2}(t,y))=(a\sin t(|y_{1}+y_{2}|+1),a\cos t|y_{1}-y_{2}|)$ and $A$ from Example \ref{exp1}. We again consider 3 standard norms on $\R^2$ to estimate $g_1$, $g_2$ and $M$.

\textbf{Case 1.} $\|y\|_1=|y_1|+|y_2|$. Then we derive
$$
\begin{gathered}
\|g(t,y)\|_{1}=|a|\left(|\sin t(|y_{1}+y_{2}|+1)|+|\cos t|y_{1}-y_{2}||\right)\\
\le |a|+|a|\left(|\sin t|+|\cos t|\right)\|y\|_1\doteq|a|+1.41421|a|\|y\|_1.
\end{gathered}
$$
Hence $g_1=|a|$, $g_2\doteq1.41421|a|$ and $M$ is given in the case 1 of Example \ref{exp1}. So condition \eqref{bf2} holds if
\begin{equation}\label{a12}
|a|<0.406656.
\end{equation}

\textbf{Case 2.} $\|y\|_\infty=\max\{|y_1|,|y_2|\}$. Then we derive
$$
\begin{gathered}
\|g(t,y)\|_{\infty}=|a|\max\left\{|\sin t(|y_{1}+y_{2}|+1)|,|\cos t|y_{1}-y_{2}||\right\}\\
\le |a|\max\left\{|\sin t|y_{1}+y_{2}||+1,|\cos t|y_{1}-y_{2}||+1\right\}\\
\le |a|+2|a|\max\left\{|\sin t|,|\cos t|\right\}\|y\|_\infty=|a|+2|a|\|y\|_\infty.
\end{gathered}
$$
Hence $g_1=|a|$, $g_2=2|a|$ and $M$ is given in the case 2 of Example \ref{exp1}. So condition \eqref{bf2} holds if
\begin{equation}\label{a22}
|a|<0.335414.
\end{equation}

\textbf{Case 3.} $\|y\|_2=\sqrt{y_1^2+y_2^2}$. Then we derive
$$
\begin{gathered}
\|g(t,y)\|_{2}=|a|\sqrt{\sin^2t(|y_{1}+y_{2}|+1)^2+\cos^2t(y_{1}-y_{2})^2}\\
\le |a|\sqrt{2\sin^2t+2\sin^2t(y_{1}+y_{2})^2+\cos^2t(y_{1}-y_{2})^2}\\
\le \sqrt{2}|a|+|a|\sqrt{2\sin^2t(y_{1}+y_{2})^2+\cos^2t(y_{1}-y_{2})^2}\\
\le \sqrt{2}|a|+\sqrt{2}|a|\max\left\{\sqrt{2}|\sin t|,|\cos t|\right\}\|y\|_2\le \sqrt{2}|a|+2|a|\|y\|_2.
\end{gathered}
$$
Hence $g_1=\sqrt{2}|a|$, $g_2=2|a|$ and $M$ is given in the case 3 of Example \ref{exp1}. So condition \eqref{bf2} holds if
\begin{equation}\label{a32}
|a|<0.35553.
\end{equation}
Consequently, the best estimate from \eqref{a12}, \eqref{a22} and \eqref{a32} is \eqref{a12} corresponding the the norm $\|\cdot\|_1$, when there is a $\pi$-antiperiodic solution $y(t)$ which is nonconstant.
\end{exmp}

Related results to Examples \ref{exp1}, \ref{exp2} and \ref{exp3} are given in \cite{BF,FT,FNO}, but our analysis is different since we focus in these examples on finding optimal norms.

\section{Extension to mild solutions}\label{s4}

In this section, we extend the above results of \eqref{e4} by assuming the following condition
\begin{itemize}
\item[(A2)] $A$ is an infinitesimal generator of a strongly continuous semigroup of bounded linear operators $\{S(t)\}_{t\in\R_+}$ in $X$ \cite{P}.
\end{itemize}
We know \cite{P} that there are constants $Q\ge1$ and $\gamma\in\R$ such that
\begin{equation}\label{gr1}
\|S(t)\|\le Qe^{\gamma t},\quad t\ge0.
\end{equation}
So we consider in the above definitions and assumptions $\R_+=[0,\infty)$ instead of $\R$. Then we look for a $(\om,c)$-periodic mild solution $y(t)\in C(\R_+,X)$ of \eqref{e4}, i.e., $y$ solving the equation
\begin{equation}\label{mild}
y(t)=S(t)y_0+\int_0^tS(t-s)g(s,y(s))ds
\end{equation}
for some $y_0\in X$ and any $t\in\R_+$.
\begin{lemma}\label{lem3}
Suppose (C1). If there is $y\in Z$ satisfying \eqref{b1} and \eqref{mild}, then the unique extension $y\in\Upsilon_{\om,c}$ of $y(t)$ to $\R_+$ satisfies \eqref{mild} on the whole $\R_+$.
\end{lemma}
\begin{proof}
The unique extension $y(t)$ is given in the proof of Lemma \ref{lem1}. Setting
$$
z(t)=S(t)y_0+\int_0^tS(t-s)g(s,y(s))ds
$$
and using
$$
S(\om)y_0+\int_0^{\om}S(t-s)g(\om,y(s))ds=y(\om)=cy(0)=cy_0,
$$
we derive for $t\in \R_+$
$$
\begin{gathered}
z(t+\om)=S(t+\om)y_0+\int_0^{t+\om}S(t+\om-s)g(s,y(s))ds\\
=S(t)S(\om)y_0+\int_0^{\om}S(t+\om-s)g(s,y(s))ds+\int_\om^{t+\om}S(t+\om-s)g(s,y(s))ds\\
=S(t)\left(S(\om)y_0+\int_0^{\om}S(\om-s)g(s,y(s))ds\right)\\
+\int_0^{t}S(t-s)g(s+\om,y(s+\om))ds=cS(t)y_0+\int_0^{t}S(t-s)g(s+\om,cy(s))ds\\
=cS(t)y_0+c\int_0^{t}S(t-s)g(s,y(s))ds=cz(t).
\end{gathered}
$$
Consequently it holds $z\in\Upsilon_{\om,c}$. But $z(t)=y(t)$ on $[0,\om]$ and the $(\om,c)$-periodic extension of $y(t)$ is unique, so $z(t)=y(t)$ on $\R_+$, this means that $y(t)$ satisfies \eqref{mild} on $\R_+$. The proof is finished.
\end{proof}

By the above lemma to find a $(\om,c)$-periodic mild solution $y(t)\in C(\R_+,X)$ of \eqref{e4} is equivalent for searching $y\in Z$ satisfying \eqref{b1} and \eqref{mild}. Then we get
$$
(cI-S(\om))y_0=\int_0^{\om}S(\om-s)g(s,y(s))ds.
$$
So extending (A1) to
\begin{itemize}
\item[(A3)] $c\notin\sigma(S(\om))$.
\end{itemize}
we get
$$
\begin{gathered}
y(t)=S(t)(cI-S(\om))^{-1}\int_0^{\om}S(\om-s)g(s,y(s))ds+\int_0^tS(t-s)g(s,y(s))ds\\
=\int_0^tS(t-s)\left((cI-S(\om))^{-1}S(\om)+I\right)g(s,y(s))ds\\
+\int_t^\om S(\om+t-s)(cI-S(\om))^{-1}g(s,y(s))ds\\
=\int_0^tcS(t-s)(cI-S(\om))^{-1}g(s,y(s))ds\\
+\int_t^\om S(\om+t-s)(cI-S(\om))^{-1}g(s,y(s))ds
\end{gathered}
$$
giving
\begin{equation}\label{mildb}
y(t)=\int_0^\om G(t,s)g(s,y(s))ds
\end{equation}
for
$$
G(t,s)=\begin{cases}cS(t-s)(cI-S(\om))^{-1}&\quad 0\le s\le t\le\om,\\
                    S(\om+t-s)(cI-S(\om))^{-1}&\quad 0\le t<s\le\om.
       \end{cases}
$$
Next, using \eqref{gr1} for any $f\in Z$, we derive
$$
\begin{gathered}
\left\|\int_0^\om G(t,s)f(s)ds\right\|\le \int_0^\om \|G(t,s)f(s)\|ds\\
\le \int_0^t\|cS(t-s)(cI-S(\om))^{-1}f(s)\|ds\\
+\int_t^\om\|S(\om+t-s)(cI-S(\om))^{-1}f(s)\|ds\\
\le Q\|(cI-S(\om))^{-1}\|\|f\|_0\left(|c|\int_0^te^{\ga(t-s)}ds
+\int_t^\om e^{\ga(\om+t-s)}ds\right)\\
=Q\|(cI-S(\om))^{-1}\|\|f\|_0\left(|c|\frac{e^{\ga t}-1}{\ga}+\frac{e^{\ga\om}-e^{\ga t}}{\ga}\right)\le U\|f\|_0.
\end{gathered}
$$
Hence we arrive at
$$
\left\|\int_0^\om G(\cdot,s)f(s)ds\right\|_0\le U\|f\|_0
$$
for any $f\in Z$ and for
\begin{equation}\label{U}
U=\begin{cases} \frac{Q(e^{\ga\om}-1)}{\ga}\|(cI-S(\om))^{-1}\|\max\{|c|,1\} & \ga\ne0\\
                Q\om\|(cI-S(\om))^{-1}\|\max\{|c|,1\} & \ga=0
\end{cases}
\end{equation}
 Now we can extend Theorem \ref{th1} as follows
\begin{theorem}\label{th3}
Suppose (A2), (A3) and consider \eqref{e4} under conditions (C1) and (C2).
If
\begin{equation}\label{bf3}
LU<1,
\end{equation}
then \eqref{e4} has a unique $(\om,c)$-periodic mild solution $y$ satisfying
\begin{equation}\label{est1b}
\|y\|_0\le\frac{U\|g(\cdot,0)\|_0}{1-LU},
\end{equation}
where $U$ is given by \eqref{U}.
\end{theorem}

\begin{rem}\label{rem1}
\rm Note under (A2) and (C2), there is a unique mild solution of \eqref{e4} on $\R_+$ for any $y_0\in Y$ depending continuously on $y_0$ \cite{P}.
\end{rem}

If we consider instead of (A2) the following assumption
\begin{itemize}
\item[(A4)] $A$ is an infinitesimal generator of a strongly continuous group of bounded linear operators $\{S(t)\}_{t\in\R}$ in $X$ \cite{P}
\end{itemize}
then we replace $\R_+$ to $\R$ in the above results.

\begin{exmp}\label{exam1}
\rm We consider a nonlinear heat equation with a forcing
\begin{equation}\label{sg}
\begin{gathered}
y_t-y_{xx}+\frac{y^3}{2(y^2+1)}=a\sin t,\quad 0\ne a\in\R\\
y(0,t)=y(\pi,t)=0,\quad t\ge0,\, x\in[0,\pi].
\end{gathered}
\end{equation}
Now we have $c=-1$, $\om=\pi$, a real $X=L^2(0,\pi)$ with a norm $\|y\|=\sqrt{\int_0^\pi y(t)^2dt}$ and $Ay=y_{xx}$ with
$$
D(A)=\{y\in X\mid y',y''\in X,\, y(0)=y(\pi)=0\}.
$$
Clearly (C1) holds. It is well-known that the sequence $\left\{\sqrt{\frac{2}{\pi}}\sin kx\right\}_{k\in\N}$ is an orthonormal basis of $X$. If
$$
y_0(x)=\sum_{k\in\N}y_{0k}\sqrt{\frac{2}{\pi}}\sin kx\in X,\quad \|y_0\|=\sqrt{\sum_{k\in\N}y_{0k}^2},
$$
then
\begin{equation}\label{st}
S(t)y_0=\sum_{k\in\N}e^{-k^2t}y_{0k}\sqrt{\frac{2}{\pi}}\sin kx \Rightarrow \|S(t)y_0\|=\sqrt{\sum_{k\in\N}e^{-2k^2t}y_{0k}^2}\le e^{-t}\|y_0\|.
\end{equation}
Hence $Q=1$ and $\ga=-1$. Moreover, $\sigma(S(\pi))=\{e^{-\pi k^2}\}_{k\in\N}$, so (A3) is verified. Next, we derive
$$
(-I-S(\pi))^{-1}y_0=-\sum_{k\in\N}\frac{1}{1+e^{-\pi k^2}}y_{0k}\sqrt{\frac{2}{\pi}}\sin kx \Rightarrow \|(-I-S(\pi))^{-1}\|=1.
$$
Hence $U=1-e^{-\pi}$ (see \eqref{U}). Now, the function $y\to\frac{y^3}{2(y^2+1)}$ has a Lipschitz constant $\frac{9}{16}$, so we have $L=\frac{9}{16}$ in (C2). Thus
$$
LU=\frac{9(1-e^{-\pi })}{16}\doteq0.538192<1
$$
and \eqref{bf3} is verified. By Theorem \ref{th3}, \eqref{sg} has a unique $\pi$-antiperiodic mild solution for any $0\ne a\in\R$ satisfying
$$
\max_{t\in[0,\pi]}\|y(\cdot,t)\|\le\frac{16\left(e^{\pi }-1\right)\sqrt{\pi}}{9+7e^{\pi}}|a|\doteq3.67222|a|.
$$
\end{exmp}

\begin{exmp}
\rm We consider a nonlinear Schr\"odinger equation with a forcing
\begin{equation}\label{sch}
\begin{gathered}
\frac{1}{\imath}y_t-y_{xx}+\frac{|y|^2y}{5(|y|^2+1)}=a(1+\sin^2x)e^{\frac{t}{4}\imath},\quad 0\ne a\in\C\\
y(x,t)=y(x+2\pi,t)=0,\quad t\ge0,\, x\in\R.
\end{gathered}
\end{equation}
Now we have $c=e^{\frac{\pi}{4}\imath}=\frac{1+\imath}{\sqrt{2}}$, $\om=\pi$, a complex $X=L^2(0,2\pi)$ with a norm $\|y\|=\sqrt{\int_0^{2\pi}|y(t)|^2dt}$ and $Ay=\imath y_{xx}$ with
$$
D(A)=\{y\in X\mid y',y''\in X,\, y(0)=y(\pi),\, y'(0)=y'(\pi)\}.
$$
Clearly (C1) holds. It is well-known that the sequence $\left\{\sqrt{\frac{1}{2\pi}}e^{kx\imath}\right\}_{k\in\Z}$ is an orthonormal basis of $X$. If
$$
y_0(x)=\sum_{k\in\Z}y_{0k}\sqrt{\frac{1}{2\pi}}e^{kx\imath}\in X,\quad \|y_0\|=\sqrt{\sum_{k\in\Z}y_{0k}^2},
$$
then
$$
S(t)y_0=\sum_{k\in\Z}e^{-k^2t\imath}y_{0k}\sqrt{\frac{1}{2\pi}}e^{kx\imath} \Rightarrow \|S(t)y_0\|=\sqrt{\sum_{k\in\Z}y_{0k}^2}=\|y_0\|.
$$
Hence $Q=1$ and $\ga=0$. Note that now (A4) holds. Moreover, $\sigma(S(\pi))=\{\pm1\}$, so (A3) is verified. Next, we derive
$$
\begin{gathered}
\left(\frac{1+\imath}{\sqrt{2}}I-S(\pi)\right)^{-1}y_0=\sum_{k\in\Z}\frac{1}{\frac{1+\imath}{\sqrt{2}}-e^{-\pi k^2\imath}}y_{0k}\sqrt{\frac{1}{2\pi}}e^{kx\imath}\\
\Rightarrow \left\|\left(\frac{1+\imath}{\sqrt{2}}I-S(\pi)\right)^{-1}\right\|=\sqrt{1+\frac{1}{\sqrt{2}}}\doteq1.30656.
\end{gathered}
$$
Hence $U=\pi\sqrt{1+\frac{1}{\sqrt{2}}}\doteq4.10469$ (see \eqref{U}). Now, the function $H(y)=\frac{|y|^2y}{5(|y|^2+1)}$, $H : \C\to\C$ has a derivative
$$
DH(y)v=\frac{(|y|^4+2|y|^2)v+y^2\bar v}{5(|y|^2+1)^2}
$$
for $\bar v$ denoting the complex conjugate of $v\in\C$. Hence
$$
|DH(y)v|\le\frac{|y|^4+3|y|^2}{5(|y|^2+1)^2}|v|\le\frac{9}{40}|v|.
$$
Hence $H$ has a Lipschitz constant $\frac{9}{40}$, so we have $L=\frac{9}{40}$ in (C2). Thus
$$
LU=\frac{9\pi}{40}\sqrt{1+\frac{1}{\sqrt{2}}}\doteq0.923555<1
$$
and \eqref{bf3} is verified. By Theorem \ref{th3}, \eqref{sch} has a unique $\left(\pi,\frac{1+\imath}{\sqrt{2}}\right)$-periodic mild solution for any $0\ne a\in\C$ satisfying
$$
\max_{t\in[0,2\pi]}\|y(\cdot,t)\|\le\frac{20\sqrt{38\left(2+\sqrt{2}\right)}\pi^{3/2}}{80-9\sqrt{2\left(2+\sqrt{2}\right)}\pi}|a|\doteq207.421|a|.
$$
\end{exmp}

Finally, we extend Theorem \ref{th2} as follows.

\begin{theorem}\label{th4}
Assume (A2), (A3), (C1), (C2), (C3) along with
\begin{itemize}
\item[(A5)] $S(t)$ is compact for any $t>0$.
\end{itemize}
If
\begin{equation}\label{eq1}
Q\|(cI-S(\om))^{-1}\|e^{\ga\om}(e^{Qg_2\om}-1)<1,
\end{equation}
then \eqref{e4} has a $(\om,c)$-periodic mild solution.
\end{theorem}
\begin{proof}
Recalling Remark \ref{rem1}, we denote by $y(y_0,t)$, $t\in\R_+$ the unique mild solution of \eqref{e4} and introduce a mapping
\begin{equation}\label{pon}
P(y_0)=(cI-S(\om))^{-1}\int_0^\om S(\om-s)g(s,y(y_0,s))ds.
\end{equation}
Note $y_0=P(y_0)$ is equivalent to $y(y_0,\om)=cy_0$, so fixed points of $P$ determine $(\om,c)$-periodic mild solutions of \eqref{e4}. Next \eqref{gr1}, \eqref{mild} and (C3) imply
$$
\begin{gathered}
\|y(y_0,t)\|\le Qe^{\ga t}\|y\|_0+Q\int_0^te^{\gamma(t-s)}\left(g_1+g_2\|y(y_0,s)\|\right)ds\\
=Qe^{\ga t}\|y\|_0+Qg_1e^{\gamma t}\int_0^te^{-\gamma s}ds+Qg_2e^{\gamma t}\int_0^t\|y(y_0,s)\|e^{-\ga s}ds\\
\le Qe^{\ga t}\|y\|_0+Qg_1e^{\gamma t}\omega e^{|\gamma|\omega}+Qg_2e^{\gamma t}\int_0^t\|y(y_0,s)\|e^{-\ga s}ds
\end{gathered}
$$
which gives
$$
\|y(y_0,t)\|e^{-\ga t}\le Q\|y_0\|+Qg_1\om e^{|\ga|\om}+Qg_2\int_0^t\|y(y_0,s)\|e^{-\ga s}ds
$$
for $t\in[0,\om]$, so Gronwall inequality gives
$$
\|y(y_0,t)\|\le Q\left(\|y_0\|+g_1\om e^{|\ga|\om}\right)e^{(Qg_2+\ga)t}\, \forall t\in[0,\om].
$$
Then \eqref{pon} has the estimate
$$
\begin{gathered}
\|P(y_0)\|\le \|(cI-S(\om))^{-1}\|\int_0^\om\|S(\om-s)\|(g_1+g_2\|y(y_0,s)\|)ds\\
\le Q\|(cI-S(\om))^{-1}\|\int_0^\om e^{\ga(\om-s)}\Big(g_1+g_2 Q\left(\|y_0\|+g_1\om e^{|\ga|\om}\right)e^{(Qg_2+\ga)s}\Big)ds\\
\le Q\|(cI-S(\om))^{-1}\|\left(g_1e^{|\ga|\om}\om+g_2e^{\ga\om}Q\left(\|y_0\|+g_1\om e^{|\ga|\om}\right)\int_0^\om e^{Qg_2s}ds\right)\\
= Q\|(cI-S(\om))^{-1}\|\left(g_1e^{|\ga|\om}\om+e^{\ga\om}\left(\|y_0\|+g_1\om e^{|\ga|\om}\right)(e^{Qg_2\om}-1)\right)\\
= Q\|(cI-S(\om))^{-1}\|e^{\ga\om}(e^{Qg_2\om}-1)\|y_0\|\\
+ Q\|(cI-S(\om))^{-1}\|\left(g_1e^{|\ga|\om}\om+e^{\ga\om}g_1\om e^{|\ga|\om}(e^{Qg_2\om}-1)\right).
\end{gathered}
$$
Hence by \eqref{eq1} and taking $y_0\in X$ such that
$$
\|y_0\|\le \Xi=\frac{Q\|(cI-S(\om))^{-1}\|\left(g_1e^{|\ga|\om}\om+e^{\ga\om}Qg_1\om e^{|\ga|\om}(e^{Qg_2\om}-1)\right)}{1- Q\|(cI-S(\om))^{-1}\|e^{\ga\om}(e^{Qg_2\om}-1)},
$$
we get $\|P(y_0)\|\le \Xi$, i.e. $P : B(\Xi)\to B(\Xi)$. We already know that $P$ is continuous. Now we shaw that $P$ is also compact. For any $n\in\N$, $n>\frac{1}{\om}$, we set
$$
P_n(y_0)=(cI-S(\om))^{-1}\int_0^{\om-n^{-1}}S(\om-s)g(s,y(y_0,s))ds.
$$
Then
$$
P_n(y_0)=S(n^{-1})(cI-S(\om))^{-1}\int_0^{\om-n^{-1}}S(\om-n^{-1}-s)g(s,y(y_0,s))ds.
$$
Since
$$
\begin{gathered}
\left\|(cI-S(\om))^{-1}\int_0^{\om-n^{-1}}S(\om-n^{-1}-s)g(s,y(y_0,s))ds\right\|\le Q\|(cI-S(\om))^{-1}\|\\
\int_0^{\om-n^{-1}} e^{\ga(\om-n^{-1}-s)}\Big(g_1+g_2 Q\left(\Xi+g_1\om e^{|\ga|\om}\right)e^{(Qg_2+\ga)s}\Big)ds\\
\le Q\|(cI-S(\om))^{-1}\|
\om e^{|\ga|\om}\Big(g_1+g_2 Q\left(\Xi+g_1\om e^{|\ga|\om}\right)e^{(Qg_2+|\ga|)\om}\Big)
\end{gathered}
$$
for any $y_0\in B(\Xi)$, by (A5), $P_n(B(\Xi))$ is precompact. Furthermore, we derive
$$
\begin{gathered}
\|P(y_0)-P_n(y_0)\|\le Q\|(cI-S(\om))^{-1}\|\\
\int_{\om-n^{-1}}^\om e^{\ga(\om-s)}\Big(g_1+g_2 Q\left(\Xi+g_1\om e^{|\ga|\om}\right)e^{(Qg_2+\ga)s}\Big)ds\\
\le Q\|(cI-S(\om))^{-1}\|
e^{|\ga|\om}\Big(g_1+g_2 Q\left(\Xi+g_1\om e^{|\ga|\om}\right)e^{(Qg_2+|\ga|)\om}\Big)n^{-1},
\end{gathered}
$$
hence $P_n\rightrightarrows P$ uniformly on $B(\Xi)$. This gives precompactness of $P(\Xi)$. Summarizing, we can apply the Schauder fixed point theorem to $P$, which finishes the proof.
\end{proof}

For illustration of Theorem \ref{th4}, we consider a parametrized version of Example \ref{exam1} in the form

\begin{exmp}
\rm We consider a nonlinear heat equation with a forcing
\begin{equation}\label{sg2}
\begin{gathered}
y_t-y_{xx}+\frac{\eta y^3}{y^2+1}=a\sin t,\quad 0\ne a\in\R\\
y(0,t)=y(\pi,t)=0,\quad t\ge0,\, x\in[0,\pi]
\end{gathered}
\end{equation}
for a parameter $\eta>0$.
Then we have $L=\frac{9\eta}{8}$ in (C2). Thus
$$
LU=\frac{9(1-e^{-\pi })}{8}\eta
$$
and \eqref{bf3} is verified for
\begin{equation}\label{est4}
0<\eta<\frac{8}{9 \left(1-e^{-\pi }\right)}\doteq0.929036.
\end{equation}
By Theorem \ref{th3}, \eqref{sg2} has a unique $\pi$-antiperiodic mild solution for any $0\ne a\in\R$ and $\eta$ satisfying \eqref{est4}.

On the other hand, for any $y\in X=L^2(0,\pi)$, we have
$$
\begin{gathered}
\left\|a\sin t-\frac{\eta y^3}{y^2+1}\right\|\le |a|\|\sin t\|+\eta\sqrt{\int_0^\pi\frac{y^6(t)}{(y^2(t)+1)^2}dt}\\
\le |a|\sqrt{\frac{\pi}{2}}+\eta\sqrt{\int_0^\pi y^2(t)dt}=\sqrt{\frac{\pi}{2}}+\eta\|y\|.
\end{gathered}
$$
Thus (C3) is verified for $g_1=|a|\sqrt{\frac{\pi}{2}}$ and $g_2=\eta$. Then \eqref{eq1} has the form
$$
e^{-\pi}(e^{\eta\pi}-1)<1,
$$
i.e.,
\begin{equation}\label{est4b}
0<\eta<\frac{\ln(1+e^\pi)}{\pi}\doteq1.01347.
\end{equation}
Next, since $e^{-k^2t}\to0$ as $k\to\infty$ for any $t>0$, by \eqref{st}, the compactness of $S(t)$, $t>0$ is clear and well-known. By Theorem \ref{th4}, \eqref{sg2} has a $\pi$-antiperiodic mild solution for any $0\ne a\in\R$ and $\eta$ satisfying \eqref{est4b} and it is unique when \eqref{est4} holds.
\end{exmp}


\begin{thebibliography}{10}
\bibitem{AF} S. Aizicovici and M. Fe\v ckan: Forced symmetric oscillations of evolution equations, \textit{Nonlinear Anal.} \textbf{64} (2006), 1621-1640.
\bibitem{AK} M. U. Akhmet and A. Kivilcim: Periodic motions generated from non-autonomous grazing dynamics, \textit{Commun. Nonlinear Sci. Numer. Simul.} \textbf{49}
(2017), 48-62.
\bibitem{AAD} N. S. Al-Islam, S. M. Alsulami and T. Diagana: Existence of weighted pseudo anti-periodic solutions to some non-autonomous differential equations,
\textit{Appl. Math. Comput.} 218 (2012), 6536-6648.
\bibitem{AGP} E. Alvarez, A. Gómez and M. Pinto: $(\omega,c)$-periodic functions and mild solutions to abstract fractional integro-differential equations, \textit{Electron. J. Qual. Theory Differ. Equ.} \textbf{16} (2018), 1-8.
\bibitem{BF} F. Battelli and M. Fe\v ckan: Heteroclinic period blow-up in certain symmetric ordinary differential equations, \textit{Z. angew. Math. Phys. (ZAMP)}
\textbf{47} (1996), 385-399.
\bibitem{CL} E. A. Coddington, N. Levison: \textit{Theory of Ordinary Differential Equations}, McGraw-Hill Inc., New York, 1955.
\bibitem{Far} M. Farkas: \textit{Periodic Motions}, Springer-Verlag, New York, 1994.
\bibitem{FT} M. Fe\v ckan, R. Ma and B. Thompson: Forced symmetric oscillations, \textit{Bull. Belg. Math. Soc. Simon Stevin} \textbf{14} (2007), 73-85.
\bibitem{FNO} D. Franco, J. J. Nieto and D. O'Regan: Anti-periodic boundary value problem for nonlinear first order ordinary differential equations,
\textit{J. Math. Inequal. Appl.} \textbf{6} (2003), 477-485.
\bibitem{H} A. Haraux: Anti-periodic solutions of some nonlinear evolution equations, \textit{Manuscr. Math.} \textbf{63} (1989), 479-505.
\bibitem{P} A. Pazy: \textit{Semigroups of Linear Operators and Applications to Partial Differential Equations}, Springer-Verlag, New York, 1983.
\bibitem{WXW} J. R. Wang, X. Xiang and W. Wei: Periodic solutions of a class of integrodifferential impulsive periodic systems with time-varying generating operators on Banach space, \textit{Electron. J. Qual. Theory Differ. Equ.} \textbf{4} (2009), 1-17.
\end{thebibliography}
\end{document}